\documentclass[12pt,reqno]{amsart}
\usepackage{amsmath}
\usepackage{amssymb}
\usepackage[left=3cm,top=3cm,right=3cm,bottom=3cm]{geometry}
\usepackage{epsfig}

\newcommand{\remove}[1]{}

\begin{document}
\newtheorem{theorem}{Theorem}[section]
\newtheorem{lemma}[theorem]{Lemma}
\newtheorem{definition}[theorem]{Definition}
\newtheorem{conjecture}[theorem]{Conjecture}
\newtheorem{proposition}[theorem]{Proposition}
\newtheorem{algorithm}[theorem]{Algorithm}
\newtheorem{corollary}[theorem]{Corollary}
\newtheorem{observation}[theorem]{Observation}
\newtheorem{problem}[theorem]{Open Problem}
\newcommand{\noin}{\noindent}
\newcommand{\ind}{\indent}
\newcommand{\al}{\alpha}
\newcommand{\om}{\omega}
\newcommand{\pp}{\mathcal P}
\newcommand{\ppp}{\mathfrak P}
\newcommand{\R}{{\mathbb R}}
\newcommand{\N}{{\mathbb N}}
\newcommand{\Z}{{\mathbb Z}}
\newcommand\eps{\varepsilon}
\newcommand{\E}{\mathbb E}
\newcommand{\Prob}{\mathbb{P}}
\newcommand{\pl}{\textrm{C}}
\newcommand{\dang}{\textrm{dang}}
\renewcommand{\labelenumi}{(\roman{enumi})}
\newcommand{\bc}{\bar c}
\newcommand{\cal}[1]{\mathcal{#1}}
\newcommand{\G}{{\cal G}}
\newcommand{\T}{{\cal T}}
\renewcommand{\P}{{\cal P}}

\newcommand{\bel}[1]{\be\lab{#1}}
\newcommand{\ee}{\end{equation}}
\newcommand{\be}{\begin{equation}}
 \newcommand\eqn[1]{(\ref{#1})}

\title{Chasing robbers on random 
geometric graphs---an alternative approach}

\author{Noga Alon}
\address{Sackler School of Mathematics and Blavatnik School of 
Computer Science, Tel Aviv University, Tel Aviv 69978, 
Israel and Institute for Advanced Study, Princeton, New Jersey, 08540, USA}
\email{\tt nogaa@tau.ac.il}

\author{Pawe\l{} Pra\l{}at}
\address{Department of Mathematics, Ryerson University, Toronto, ON, Canada, M5B 2K3}
\email{\texttt{pralat@ryerson.ca}}

\thanks{
The first author acknowledges support by an ERC Advanced grant, 
by a USA-Israeli
BSF grant, and by the Hermann Minkowski Minerva Center for
Geometry at
Tel Aviv University.
The second author acknowledges support from NSERC 
and Ryerson University. 
}


\maketitle

\begin{abstract}
We study the vertex pursuit game of \emph{Cops and
Robbers}, in which cops try to capture a robber on the vertices of the
graph. The minimum number of cops required to win on a given graph $G$ is
called the cop number of $G$. 
We focus on $\G_{d}(n,r)$, a random geometric
graph in which $n$ vertices are chosen uniformly at
random and independently from $[0,1]^d$, and two vertices are adjacent
if the Euclidean distance between them is at
most $r$. The main result is that if $r^{3d-1}>c_d \frac{\log
n}{n}$ then the cop number is $1$ with probability that tends to
$1$ as $n$ tends to infinity. The case $d=2$ was proved
earlier and independently in \cite{bdfm}, using a different
approach. Our method provides a tight $O(1/r^2)$ upper bound
for the number of rounds needed to catch the robber.
\end{abstract}

\section{Introduction\label{intro}}

The game of \emph{Cops and Robbers}, introduced independently by
Nowa\-kowski  and Winkler~\cite{nw} and Quilliot~\cite{q} almost thirty
years ago, is played on a fixed graph $G$. We will always assume that $G$
is undirected, simple, and finite.  There are two players, a set of $k$
\emph{cops}, where $k\ge 1$ is a fixed integer, and the \emph{robber}.
The cops begin the game by occupying any set of $k$ vertices (in
fact, for a connected $G$, their initial position  does not matter).
The robber then chooses a vertex, and the cops and robber move in
alternate rounds. The players use edges to move from vertex to vertex.
More than one cop is allowed to occupy a vertex, and the players may
remain on their current positions. The players know each others current
locations.  The cops win and the game ends if at least one of the cops
eventually occupies the same vertex as the robber;  otherwise, that is,
if the robber can avoid this indefinitely, he wins. As placing a cop on
each vertex guarantees that the cops win, we may define the \emph{cop
number}, written $c(G)$, which is the minimum number of cops needed to
win on $G$. The cop number was introduced by Aigner and Fromme~\cite{af}
who proved (among other things) that if $G$ is planar, then $c(G)\leq 3$.
For more results on vertex pursuit games such as \emph{Cops and Robbers},
the reader is directed to the  surveys on the subject~\cite{al,ft,h}
and the monograph~\cite{bn}.
 The most important  open problem in this area is Meyniel's conjecture
(communicated by Frankl~\cite{f}). It states that $c(n) = O(\sqrt{n})$,
where $c(n)$ is the maximum of $c(G)$ over all $n$-vertex connected
graphs.  If true, the estimate is best possible as one can construct a
graph based on the finite projective plane with the cop number of order at
least  $\Omega(\sqrt{n})$.  Up until recently, the best known upper bound
of $O(n \log \log n / \log n)$ was given in~\cite{f}. This was improved
to $c(n) = O(n/\log n)$ in~\cite{eshan}. Today we know that the cop number
is at most $n 2^{-(1+o(1))\sqrt{\log_2 n}}$ (which is still $n^{1-o(1)}$)
for any connected graph on $n$ vertices (a result obtained independently
by Lu and Peng~\cite{lp} and Scott and Sudakov~\cite{ss}, see also
~\cite{am, fkl}
for some extensions).  If one looks for counterexamples
for Meyniel's conjecture it is natural to study first the cop number
of random graphs. Recent years have witnessed significant interest in
the study of random graphs from that perspective~\cite{bkl, bpw, lp2,
p} confirming that, in fact, Meyniel's conjecture holds asymptotically
almost surely for binomial random graphs~\cite{pw} as well as for random
$d$-regular graphs~\cite{pw2}.

\bigskip

In this note we consider a \emph{random geometric graph} $\G_d(n,r)$
which is defined as a random graph with vertex set $[n]=\{1,2,\dots, n\}$
in which $n$ vertices are chosen uniformly at random and independently
from $[0,1]^d$, and a pair of vertices within Euclidean distance $r$
appears as an edge---see, for example, the monograph~\cite{pen}.

As typical in
random graph theory, we shall consider only asymptotic properties of
$\G_d(n,r)$ as $n\rightarrow \infty$, where $r=r(n)$
may and usually does depend on $n$. 
We say that an event in a probability space
holds \emph{asymptotically almost surely} (\emph{a.a.s.}) if its
probability tends to one  as $n$ goes to infinity.

\section{The result and its proof}

We prove the following result. 

\begin{theorem}\label{t11}
There exists an absolute constant $c_2>0$ so that 
if $r^5 > c_2 \frac {\log n}{n}$ then a.a.s.\ $c(\G_2(n,r))=1$.
\end{theorem}

The same result was obtained earlier and
independently  in~\cite{bdfm} but the proof
presented here is quite different, provides a tight $O(1/r^2)$ 
bound for the number  of rounds required to catch the robber, and can be generalized to higher dimensions.
In the proof we describe
a strategy for the cop that is a winning one a.a.s. In~\cite{bdfm},
the known necessary and sufficient condition for a graph to be cop-win
(see~\cite{nw} for more details) is used; that is, it is shown that the
random geometric graph is what is called dismantlable a.a.s.
\bigskip

Essentially the same proof we provide here gives the following.
\begin{theorem}\label{t12}
For each fixed $d>1$ there exists a constant $c_d>0$ so that if $r^{3d-1} > c_d \frac {\log n}{n}$ then a.a.s.\ $c(\G_d(n,r))=1$.
\end{theorem}

In all dimensions the proof gives that a.a.s.\ the cop can win in
$O(1/r^2)$ steps and, as we mention below, this is tight; namely, a.a.s.\
the robber can ensure not to be caught in less steps. Therefore, the
capture time for this range of parameters is $\Theta(1/r^2)$ a.a.s. We
make no attempt to optimize the absolute constants in all arguments below,
aiming to propose an argument which is as simple as possible.

\bigskip

In order to prove Theorem~\ref{t11} it is convenient to describe first a
cleaner proof of the corresponding result for the continuous (infinite)
graph $\G_2(r)$ whose vertices are all of the points of $[0,1]^2$,
where two of them are adjacent if and only if their distance is at most
$r$. This is a natural variant of the well-known problem of \emph{the
Lion and the Christian} in which (perhaps surprisingly) the Christian
(counterpart of the robber in our game) has a winning strategy; see,
for example,~\cite{bela_coffee} for more details. In our game, the
cop (counterpart of the lion) has a winning strategy. This is a
essentially a known result 
\cite{from_Alan33, from_alan20}, but the proof
described here differs from the known ones and, crucially for
us, can be easily modified to yield a proof of Theorem~\ref{t11}
and Theorem \ref{t12}.

\begin{theorem}\label{t13}
$c(\G_2(r))=1$ for any $r>0$.
\end{theorem}
\begin{proof}
We show that $c(\G_2(r))=1$ for any $r>0$, by describing a winning
strategy for the cop. In the first step, the cop places himself at
the center $O$ of $[0,1]^2$. After each move of the robber, when he
is located at a point $R$, the cop catches him if he can (that is,
if the distance between him and the robber is at most $r$); otherwise,
he moves to a point $C$ that lies on the segment  $OR$, making sure his
distance from the robber is at least, say, $r^2/100$. Moreover, we will
show that the cop can do it and also ensure that in each step the square
of the distance  between the location of the cop and  $O$ increases by
at least $r^2/5$. As this square distance cannot be more than $1/2$,
this implies that the cop catches the robber in at most $O(1/r^2)$ steps.

\begin{figure}[htbp]
\begin{center}
\includegraphics[width=1in]{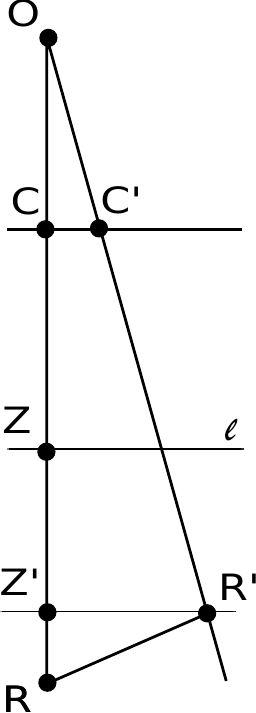}
\end{center}
\caption{Catching the bad guy on $\G_2(r)$.}\label{fig1}
\end{figure}

Here is the proof showing that the cop can indeed achieve the above in each step of the game. Suppose that the cop is located at $C$ and the robber at $R$, where $C$ lies on $OR$ (and the distance between $C$ and $R$ is at least $r^2/100$). Trivially, the cop can ensure this will be the case after his first move (unless the robber gives up prematurely without fighting and starts the game too close to the cop). A step, now, consists of a move of the robber from $R$ to $R'$ followed by a move of the cop from $C$ to $C'$. Let $Z$ denote the midpoint of $CR$ and let $\ell$ be the line through $Z$ perpendicular to $OZ$---see Figure~\ref{fig1}. Without loss of generality, choose a coordinate system so that $OR$ is a vertical line (and hence $\ell$ is a horizontal one), and assume $R$ (as well as $C$ and $Z$) are below $O$. Note that $R'$, the new location of the robber, may be assumed to be below the line $\ell$, since otherwise the distance between $C$ and $R'$ is at most the distance between $R$ and $R'$, meaning that the cop can catch the robber, winning the game. Suppose then that $R'$ is below $\ell$ and let $C'$ be the intersection point of the horizontal line through $C$ with the line $OR'$. Let also $Z'$ denote the intersection point of the horizontal line through $R'$ with the line containing $OR$---see Figure~\ref{fig1} one more time.  Now it is easy to see that $RR' \geq  R'Z' > CC'$, as the triangle $RR'Z'$ is a right-angle triangle and the two triangles $OCC'$ and $OZ'R'$ are similar. Hence, the cop may move to $C'$ if he decides to do so, as $CC'<RR' \leq r$. Consider the following two possible cases.

\smallskip
\noindent
{\bf Case 1:}\, $CC' >r/2$. In this case if the cop moves to $C'$ then its square distance to $O$ increases by $CC'^2 > r^2/4$. If $C'$ is too close to $R'$, we shift him towards $O$ slightly (that is, by less than $r^2/100$) to ensure the distance between $C'$ and $R'$ is at least $r^2/100$. Note that such a shift decreases the square distance from $O$ by less than $2r^2/100=r^2/50$, hence the square distance still increases by  at least $r^2/4-r^2/50 >r^2/5$. Thus, in this case the cop can make a step as required.

\smallskip
\noindent
{\bf Case 2:}\,  $|CC'| \leq r/2$. In this case the cop can move to $C'$ and then walk along the line $OR'$ at least distance $r/2$ towards $R'$ (without passing it, since otherwise the game ends and the cop wins). As clearly $OC' \geq OC$, in this case the cop increases its distance from $O$ by more than $r/2$ and hence its square distance by more than $r^2/4$. As before, it may be the case that he gets too close to $R'$ and then he backups slightly by less than  $r^2/100$, which is still fine.

\smallskip
This shows that in $\G_2(r)$ the cop can indeed increase its square distance  from $O$ by at least $r^2/5$ in each step (which is not a winning step ending the game), staying on the segment connecting the center and the robber (and being closer to the center than the robber). This implies that the game ends with a cop win in $O(1/r^2)$ steps. 
\end{proof}

Modifying the above argument to get a winning strategy for $\G_2(n,r)$
(a.a.s.) is not too difficult. The cop will follow essentially the same
strategy, but will always place himself at a vertex of the graph which
is sufficiently close to where he wants to be in the continuous game. 

\smallskip
More precisely, for each point $X$ of the unit square whose distance
from the center $O$ is at least $r/2$ (just to ensure that the
triangle $T(x)$ defined below will indeed be well
defined; in our argument this will always be the case) and whose distance
from the boundary is at least $r^2/10^3$ (again, in our argument
this will
always be the case), we define an isosceles triangle $T(X)$ 
as follows. One vertex is $X$, and the segment of length $r^2/100$ on
the line $OX$ starting at $X$ (and going towards $O$) is the height of
$T(X)$. The base is orthogonal to it and of length $r^3/10^5$. Despite
the fact that there are infinitely many triangles, it is not difficult
to show that if the area of such a triangle is large enough, then a.a.s.\
$\G_2(n,r)$ contains a vertex inside each such triangle.  

\begin{lemma}\label{lem:triangles}
There exists an absolute constant $c > 0$ so that a.a.s.\ every
triangle $T(X)$ contains a vertex of $\G_2(n,r)$, provided $r^5 >
c \frac {\log n}{n}$. 
\end{lemma}
\begin{proof}
Let us start with a fixed collection $F$ of $O((1/r)^6)$ rectangles,
each of area $\Omega(r^5)$, so that every triangle $T(X)$ fully contains
at least one of these rectangles. To do so, for each point $Y$ in an
$10^6 r^3$ by $10^6 r^3$ grid in the unit square take the rectangle
of width $r^3/10^6$ and height $r^2/10^6$ in which $Y$ is the midpoint of
the edge of length $10^6 r^3$ and the other edge is in direction $YO$. It
is clear that every $T(X)$ under consideration ($X$ not too close to $O$
nor to the boundary) fully contains at least one such a rectangle.

In order to complete the proof it is enough to show that a.a.s.\ 
each rectangle in $F$ contains at least one vertex of $\G_2(n,r)$. 
The area of each such rectangle is $r^5/10^{12}$, and hence the
probability it contains no vertex is 
$$
\left( 1-\frac{r^5}{10^{12}} \right)^n \leq e^{-c \log n/ 10^{12}}.
$$
Since there are $O((1/r)^6) = O(n^2)$ rectangles, the desired 
result follows by the union bound for, say, $c = 10^{13}$, as
needed. 
\end{proof}

Now, let us come back to the main result of this section, since we have all necessary ingredients. 

\begin{proof}[Proof of Theorem~\ref{t11}]
Since we aim for a statement that holds a.a.s., it follows from
Lemma~\ref{lem:triangles} that we may assume that every triangle $T(X)$
contains at least one vertex. As we already mentioned, the cop plays
the continuous strategy, but whenever he wants to place himself at a
point $X$, he chooses an arbitrary vertex $x \in V$ of $T(X)$ to go
to. The line $R'x$ is now not necessarily identical  to the line $R'O$,
but the angle between them is sufficiently small to ensure that in the
computations above for the continuous case we do not lose much. That
was the reason we ensured that $R'$ and $X$ are never too close in
the  continuous algorithm, and as the triangle $T(X)$ is thin, the
angle between these two lines is smaller than $r/10^3$. 
This completes the proof. 
\end{proof}

As we already mentioned, essentially the same proof works for general
dimension. The continuous game in dimension $d$ is nearly identical to
the one in dimension $2$. In the first step, the cop places himself at
the center $O$ of $[0,1]^d$. After each move of the robber, when he
is located at a point $R$, the cop catches him if he can, otherwise,
he moves to a point $C$ that lies on the segment  $OR$, making sure
his distance from the robber is at least, say, $r^2/100$.  As in
the planar case, the cop can do it and also ensure that in each step
the square of the distance between his location and  $O$
increases by at least $r^2/5$. Indeed, since in each round the center
$O$, the location of the cop $C$, the old location  of the robber $R$
and his new location $R'$ lie in a two dimensional plane (since $O$,
$C$ and $R$ lie on a line) the analysis is identical to the planar case.
As the square distance of the cop from the center cannot exceed $d/4$,
this implies that the cop catches the robber in at most $O(d/r^2)$ steps.

In the discrete game we let $T(X)$ be a cone with height $r^2/100$ on
the line connecting the center $O$ to $X$, and basis of radius $r^3/10^5$
centered at $X$.  The probabilistic estimate given in the proof of Lemma
\ref{lem:triangles} shows that a.a.s. every such cone $T(X)$ contains a
vertex of our graph, provided $r^{3d-1} > c_d \frac {\log n}{n}$. 
We can thus repeat the arguments in the proof of the planar case 
to show that the assertion of Theorem~\ref{t12} holds.

\bigskip

Finally, note that the  robber can keep escaping for $\Omega(1/r^2)$
steps a.a.s. For simplicity, we describe the strategy for the robber for
the plane but this also holds for any dimension, for the same reason. As
before, we start with the continuous variant of the game. 
Initially, the robber places himself at distance bigger than $r$
from  the cop ensuring he is not too far from the center $O$ of the
square.  At each
step, when the robber located at $R$ has to move, he moves distance $r$
exactly in the direction perpendicular to $RC$, where the choice of the
direction (among the two options), is such that its square distance from
$O$ increases by at most $r^2$ (that is, the angle $ORR'$ is at most
$\pi/2$). This suffices for the continuous case, as it is clear that
the distance from the cop will exceed $r$ after each such step. In the
discrete case, the robber simply chooses a nearby point, making sure
his distance from the cop is at least what it would have been in the
continuous case.

\end{document}